\documentclass[11pt]{amsart}
\usepackage{amsmath}
\usepackage{hhline}
\usepackage{stackengine}
\usepackage{amsthm}
\usepackage{blkarray}
\usepackage{mathrsfs}
\usepackage{amssymb}
\usepackage{mathtools}
\usepackage{xcolor}
\usepackage{algorithm}
\usepackage{algpseudocode}
\usepackage{bbold}
\usepackage{cite}
\usepackage{latexsym}
\usepackage{booktabs}
\usepackage{bm}
\usepackage{graphicx}
\usepackage{textcomp}
\usepackage{subfigure}
\usepackage[toc,page]{appendix}
\usepackage{url}
\usepackage{hyperref}
\usepackage{multicol}
\usepackage{multirow}
\usepackage{tikz}
\usetikzlibrary{3d}
\usepackage{tikz-3dplot}
\usepackage{pgfplots}

\let\oldtextit\textit 
\renewcommand\emph[1]{\oldtextit{\color{blue}#1}}

\theoremstyle{definition}
\newtheorem{definition}{Definition}[section]
\newcommand{\tim}[1]{\textcolor{blue}{TD: #1}}
\newcommand{\kisun}[1]{\textcolor{red}{KL: #1}}

\usepackage{cleveref}

\newtheorem{corollary}[definition]{Corollary}

\newtheorem{thm}[definition]{Theorem}

\theoremstyle{remark}
\newtheorem{remark}[definition]{Remark}

\newcommand{\CC}{{\mathbb C}}
\newcommand{\RR}{{\mathbb R}}

\newcommand{\mainfilecheck}[1]{1}

\makeatletter
\def\@settitle{\begin{center}%
  \baselineskip13\p@\relax
    \Large
\@title
  \end{center}%
}
\makeatother
\title{Certified homotopy tracking using the Krawczyk method}

\author{Timothy Duff}
\address{Department of Mathematics, University of Washington, Box 354350, Seattle, WA 98195}
\email{timduff@uw.edu}
\urladdr{https://timduff35.github.io/timduff35/}

\author{Kisun Lee}
\address{School of Mathematical and Statistical Science, Clemson University, 220 Parkway Drive, Clemson, SC 29634}
\email{kisunl@clemson.edu}
\urladdr{https://klee669.github.io}

\begin{document}

\begin{abstract}
We revisit the problem of certifying the correctness of approximate solution paths computed by numerical homotopy continuation methods.
We propose a conceptually simple approach based on a parametric variant of the Krawczyk method from interval arithmetic.
Unlike most previous methods for certified path-tracking, our approach is applicable in the general setting of parameter homotopies commonly used to solve polynomial systems of equations. 
We also describe a novel preconditioning strategy and give theoretical correctness and termination results. 
Experiments using a preliminary implementation of the method indicate that our approach is competitive with specialized methods appearing previously in the literature, in spite of our more general setting.
\end{abstract}

\maketitle

\section{Introduction}

\emph{Homotopy continuation} is a popular method for finding solutions to a system of nonlinear equations. The main idea involves a system $G(x):\mathbb{C}^n\rightarrow\mathbb{C}^n$ for which we already know the solutions (points $x$ with $G(x)=0$), and tracking these solutions towards the solutions of another system $F(x):\mathbb{C}^n\rightarrow\mathbb{C}^n$ that we wish to solve. This is done by constructing a homotopy $H(x,t):\mathbb{C}^n\times [0,1]\rightarrow\mathbb{C}^n$ such that $H(x,0)=G(x)$ and $H(x,1)=F(x)$. 
In many cases of interest, $F$ and $G$ are both polynomial systems with finitely many nonsingular solutions.
The homotopy is typically constructed such that the solutions to $H(x(t),t)=0$ are implicit functions of $t,$ each represented by a smooth \emph{solution path} $x(t) : [0,1] \to \CC^n$.

To track values of a solution path $x(t)$ numerically from $t=0$ to $t=1$, it is common to use a numerical \emph{predictor-corrector method} \cite[Chapter 2.3]{SommeseWampler:2005}. When an approximation $x_0$ for a solution $x(t_0)$ to $H(x,t_0)$ is available, the path tracking proceeds when we find an approximation $x_1$ for a solution $x(t_1)$ to $H(x,t_1)$ for some $t_1>t_0$. The predictor-corrector method consists of a step constructing a rough approximation for $x(t_1)$ (a predictor step, e.g.~Euler's method) and a step refining this approximation (a corrector step, typically a variant of Newton's method).

For a system $F(x):\mathbb{C}^n\rightarrow\mathbb{C}^n$, we say that $x\in \mathbb{C}^n$ is \emph{certified} if $x$ is contained in a small region which also contains a unique solution $x^*$, and if $x$ can be refined to be arbitrarily close to $x^*$ using a finite procedure (such as Newton's method).

The main problem addressed in this paper is \emph{certified homotopy tracking}: given a homotopy $H(x,t)$ and an initial solution $x(0)$ at time $t^{(0)}=0,$ rigorously certify a sequence of approximations $x_1, \ldots , x_k$ to values of a solution path $x(t^{(1)}), \ldots , x(t^{(k)})$ at discrete time-steps $t^{(1)}< \ldots < t^{(k)}=1$ along with proving the existence and uniqueness of the solution path $x(t)$ along each interval $[t^{(i-1)}, t^{(i)}]$ (that is, constructing an interval box $I^{(i)}$ in $\mathbb{C}^n$ that contains $x(t)$ uniquely for all $t\in [t^{(i-1)},t^{(i)}]$). If this is achieved, we say that $x(t)$ is a \emph{certified solution path}.
In particular, it is not enough to certify that the final approximation $x_k$ is near some solution of $F$ using \textit{a posteriori} methods (e.g.~\cite{breiding2020certifying,burr2019effective,hauenstein2012algorithm,lee2019certifying,krawczyk1969newton}). 
We must show that the true solution approximated by $x_k$ is in fact $x(1).$

The contributions of this paper can be summarized as follows:
\begin{thm} 
\Cref{algo:krawczyk_homotopy_rec,algo:krawczyk_homotopy}, when they terminate, return certified paths for a square linear parameter homotopy.
\end{thm}
\begin{corollary}
For a given certified homotopy path $x(t)$ for a homotopy $H(x,t):\mathbb{C}^n\times[0,1]\rightarrow\mathbb{C}^n$ with $x_0$ approximating $x(0)$, $x_k$ is a certified solution to $H(x,1)$ that can be refined to $x(1)$.
\end{corollary}

There have been a number of previous studies in certified homotopy tracking.
For polynomial systems, Beltr\'{a}n and Leykin \cite{beltran2012certified,beltran2013robust} give a certified homotopy tracking algorithm based on Smale's alpha theory \cite[Chapter 8]{blum1998complexity}, mostly tuned to the ``generic'' case of total-degree homotopies involving dense polynomials.
Another noteworthy contribution~\cite{hauenstein2014posteriori} considers the special case of ``Newton homotopies'', where $H(x,t) = F(x) + (1-t)v$ for some fixed $v\in \CC^n.$

Yet another class of certified homotopy methods involves interval arithmetic.
Early work of Kearfott and Xing \cite{doi:10.1137/0731048} proposes a general solution where intervals enclosing the solution path at every time-step are constructed.
More sophisticated variants have since been proposed, e.g.~in~\cite{van2011reliable}, for the case of univariate polynomials in~\cite{xu2018approach}, and most recently in the remarkable preprint~\cite{guillemot2024validated}.
Two appealing aspects of these interval-based methods are that they (1) naturally accommodate systems $F,G$ represented as straight-line programs (also known as algebraic circuits), and (2) generally involve the Krawczyk method~\cite{krawczyk1969newton}, whose \textit{a posteriori} certificates may be easier to verify than those coming from alpha-theory.



In this paper, we propose an interval-based \emph{Krawczyk homotopy} for certified homotopy tracking.
After reviewing preliminaries in~\Cref{sec:prelim}, we consider two Krawczyk homotopy variants in~\Cref{sec:algorithms}: a base-line method (\Cref{algo:krawczyk_homotopy_rec}) illustrating main ideas, and a more effective ``tilted'' variant (\Cref{algo:krawczyk_homotopy}) based on a novel preconditioning step.
\Cref{sec:proofs} addresses correctness and termination for these variants. 
To simplify our analysis in this section, we consider only affine-linear homotopies and assume a model of computation allowing exact computation with real numbers.
In \Cref{sec:experiments}, the experimental results using a proof-of-concept implementation are presented, demonstrating favorable performance for our ``tilted'' variant.




\section{Preliminaries}\label{sec:prelim}

\subsection{Interval arithmetic}
Interval arithmetic performs conservative computation for certified results by arithmetic between intervals. Specifically speaking, for an arithmetic operator $\odot$ and two intervals $I_1$ and $I_2$, we define 
\[I_1\odot I_2:=\{x\odot y\mid x\in I_1, y\in I_2\}.\]
There are formulas for the interval version of standard arithmetic operations. 
Thus, for example
\[
[a,b] + [c,d] = [a+c, b+d].
\]
For more details, see \cite{moore2009introduction}.
These operations cannot be computed exactly when the endpoints $a,\ldots , d$ are represented in floating point, in which case the resulting intervals must be rounded outward.
Since our analysis in~\Cref{sec:proofs} assumes exact real number computation, such concerns do not play a significant role in this paper.

Although it is natural to consider the concepts of intervals with real numbers, interval arithmetic can be extended to complex numbers by introducing intervals for real and imaginary parts. 
In other words, we consider intervals $I_1=\Re(I_1)+i\Im(I_1)$ and $I_2=\Re(I_2)+i\Im(I_2)$. Then, based on interval arithmetic over $\mathbb{R}$, we may define interval arithmetic over $\CC$ as follows:
\begin{align*}
I_1+I_2&=(\Re(I_1)+\Re(I_2))+i(\Im(I_1)+\Im(I_2))\\
I_1-I_2&=(\Re(I_1)-\Re(I_2))+i(\Im(I_1)-\Im(I_2))\\
I_1\cdot I_2&=(\Re(I_1)\cdot\Re(I_2)-\Im(I_1)\cdot \Im(I_2))\\
&\quad\quad\quad\quad\quad+i(\Re(I_1)\cdot\Im(I_2)+\Im(I_1)\cdot\Re(I_2))\\
I_1/I_2 &= \frac{\Re(I_1)\cdot\Re(I_2)+\Im(I_1)\cdot\Im(I_2)}{\Re(I_2)\cdot \Re(I_2)+\Im(I_2)\cdot \Im(I_2)}\\
&\quad\quad\quad\quad\quad+i\frac{\Im(I_1)\cdot\Re(I_2)-\Re(I_1)\cdot\Im(I_2)}{\Re(I_2)\cdot \Re(I_2)+\Im(I_2)\cdot \Im(I_2)}\quad\text{if }0\not\in I_2.
\end{align*}
From now on, we consider the intervals over $\mathbb{C}$ and interval arithmetic over complex numbers unless otherwise mentioned.

Let $I=(I_1,\dots, I_n)$ be an $n$-dimensional interval box in $\mathbb{C}^n$. For a function $f:\mathbb{C}^n\rightarrow \mathbb{C}$ with $n$ variables, we define an \emph{interval extension} $\square f(I)$ of $f$ over $I$ to be an interval in $\mathbb{C}$ satisfying that
\[\square f(I)\supset \{f(x)\mid x\in I\}.\]
In other words, we need $\square f(I)$ to be an interval containing the image of $f$ on $I$. Also, for a point $x=(x_1,\dots, x_n)\in\mathbb{C}^n$, we denote by $x$ not only the point itself but also the interval box $[\Re(x),\Re(x)]+i[\Im(x),\Im(x)]$, so that $\square f(x)$ is well-defined.
For a given function $f$ and an interval box $I$, an interval extension $\square f(I)$ is not unique since interval arithmetic may return different outputs depending on how $f$ is evaluated on $I$. For polynomials, such interval extensions are obtained by interval arithmetic.

For an interval $I=[a,b]$ over $\mathbb{R}$, the \emph{width} $w(I)$ of $I$ is defined by $w(I)=b-a$. If an interval $I$ is given over $\mathbb{C}$, we define the \emph{absolute value} of $I$ by $|I|=\max\limits_{x\in I}|x|$.  For an $n$-dimensional interval box $I=(I_1,\dots, I_n)$, the \emph{max norm} is defined by $\|I\|=\max\limits_{i=1,\dots, n}|I_i|$. 
If $I=(I_1,\dots, I_n)$ is a \emph{square} $n$-dimensional interval box over $\mathbb{C}$, i.e.\  one satisfying $$w(\Re(I_1))=w(\Im(I_1))=\cdots=w(\Re(I_n))=w(\Im(I_n)),$$
then we define the \emph{radius} of $I$ by $\frac{w(\Re(I_i)}{2}$ for any $i=1,\dots,n$.
We also define an \emph{interval matrix} whose entries are given by intervals. Note that an $m\times n$ interval matrix $M$ can be considered as a set of $m\times n$ matrices whose $ij$ entry is contained in the interval $M_{ij}$. The \emph{interval matrix norm} $\|M\|$ is defined by the maximum operator norm of a matrix in $M$ under the max norm. In other words, $\|M\|=\max\limits_{A\in M}\max\limits_{x\in \mathbb{C}^n}\frac{\|Ax\|}{\|x\|}$ where $\|x\|=\max\limits_{i=1,\dots, n}|x_i|$.

\subsection{Krawczyk method}

The Krawczyk method combines interval arithmetic and the generalized Newton's method to prove the existence and uniqueness of a solution within a region for a square system of equations. Although the Krawczyk method is introduced only for real variables in most literature, we state the result in the complex setting. Subtle differences in the Krawczyk method in the complex setting are introduced and analyzed in \cite{burr2019effective}. 

Let $F:\mathbb{C}^n\rightarrow\mathbb{C}^n$ be a polynomial system. For a point $x\in \mathbb{C}^n$, an $n$-dimensional interval vector $I$ and an invertible matrix $Y$, we define the Krawczyk operator 
\[K_{x,Y}(I):= x-Y\cdot F(x)+(\boldsymbol{1}_n-Y\cdot \square JF(I))\cdot(I-x)\]
where $\boldsymbol{1}_n$ is the $n\times n$ identity matrix. We have the following theorem.

\begin{thm}\cite{krawczyk1969newton}\label{thm:Krawczyk}
Suppose that $F:\mathbb{C}^n\rightarrow\mathbb{C}^n$ is a square differentiable system with a given interval extension $\square JF(I)$ on an interval $I$. For an $n\times n$ invertible matrix $Y$ and a point $x$,
\begin{enumerate}
\item \label{thm:KrawczykExistence} if $K_{x,Y}(I)\subset I$, then $I$ contains a solution $x^\star$ of $F$, and 
\item \label{thm:KrawczykUniqueness} if additionally $\sqrt{2}\| \boldsymbol{1}_n-Y\cdot\square JF(I)\|<1$, then the solution  $x^\star$ in $I$ is unique.
\end{enumerate}
\end{thm}

Note that the first part of the theorem proves the existence of the solution in the region $I$, and the second part proves its uniqueness. The $\sqrt{2}$ factor in \Cref{thm:Krawczyk} (\ref{thm:KrawczykUniqueness}) is used for the Krawczyk method in the complex setting. When the theorem is applied to inputs over the real, satisfying $\|\boldsymbol{1}_n-Y\cdot \square JF(I)\|<1$ is sufficient to prove the uniqueness.

In the actual application of the theorem, the invertible matrix $Y$ is chosen to minimize the norm $\|\boldsymbol{1}_n-Y\cdot \square JF(I)\|$. In the absence of additional details about the system $F$, it is common to use the midpoint of a given interval $I$ as the value for $x$ and $JF(x)^{-1}$ for $Y$. Also, the interval extension $\square F(x)$ often replaces $F(x)$ since exactly evaluating $F(x)$ may not be feasible in usual cases.


\section{Algorithms}\label{sec:algorithms}

We present the algorithmic framework for certified homotopy continuation using the Krawczyk method. Two algorithms are proposed depending on the path prediction strategy. The first algorithm is a prototypical version of the Krawczyk homotopy continuation based on the constant predictor. The second algorithm adopts a preconditioning step for more sophisticated path prediction. Elaborating on each step in the first algorithm, we describe the main idea of the Krawczyk homotopy. After that, the algorithm with the preconditioning step is proposed to improve the first algorithm.
Both variants rely on three hyper-parameters which must be set in advance: initial values for the step-size $dt$ and a radius parameter $r$ controlling the sizes of interval boxes, and a scaling parameter $\lambda$ used to update these values.

In \ifthenelse{\mainfilecheck{1} > 0}{actual}{}
applications, we are often interested in systems with parameters, $F(x;p):\CC^n\times \CC^m \to \CC^n .$
The case of the homotopy $H(x,t)=0$ is a special case with $p=t$ and $m=1.$
On the other hand, systems with parameters are often solved using \emph{parameter homotopies}~\cite[Chapter 8]{SommeseWampler:2005}.
For two points $p_0,p_1\in \mathbb{C}^m$, we consider a path $p(t)$ in the parameter space such that $p(0)=p_0$ and $p(1)=p_1$. We define a parameter homotopy $H(x,t)=F(x;p(t))$. Let $x(t)$ be a solution path of homotopy $H(x,t)$. We assume that the solution path $x(t)$ is nonsingular; that is, the Jacobian $JF(x(t);p(t))$ is invertible for all $t\in [0,1]$. In a typical application of parameter homotopies, we further assume that solutions of the start system $F(x;p_0)$ are known in advance at least approximately.

The goal of the Krawczyk homotopy algorithm is to construct a finite sequence of time-steps $0=t^{(0)}< t^{(1)}<\cdots < t^{(k)}=1$ and a collection of $k$ interval boxes $I^{(1)},\dots, I^{(k)}$ contained in $\mathbb{C}^n$ such that each $I^{(i)}\times[t^{(i-1)},t^{(i)}]$ is verified to enclose only a single solution path $x(t)$ from $t=t^{(i-1)}$ to $t=t^{(i)}$. By accomplishing this goal, we obtain a certified solution to $F(x;p_1)$ through the refinement of a point in the last interval box $I^{(k)}$. We point out that each $t^{(i)}$ and $I^{(i)}$ are computed by previously obtained $t^{(i-1)}$ and $I^{(i-1)}$. In each subsection, we elaborate on steps from computing the time sequence and collection of interval boxes to finalizing the algorithm.

\subsection{Initialization step}\label{subsec:init}
    The algorithm first initializes an interval box containing the known solution and time-step to compute the next interval box. 
    For a point $x_0\in\mathbb{C}^n$ approximating a solution $x(0)$ of $H(x,0)=F(x;p_0)$, we construct an $n$-dimensional interval box $I^{(1)}$ enclosing $x_0$. In addition, we wish to have a proper $t^{(1)}\in (0,1)$ such that $I^{(1)}\times [0,t^{(1)}]$ contains the solution path $x(t)$ uniquely from $t=0$ to $t=t^{(1)}$. In general, information for an effective guess of $I^{(1)}$ and $t^{(1)}$ may not be available. Hence, the interval box $I_r$ with the midpoint $x_0$ and the radius $r$ for some $r>0$ can be a natural choice for $I^{(1)}$. Furthermore, we begin with some $dt\in (0,1)$, preferably not too small or large compared to $r$, and set $t^{(1)}=dt$.
    For a simple explanation of iterative steps in the algorithm, we introduce the notations $t_0=t^{(0)},  t_1=t^{(1)}, I=I^{(1)}$ and write $I$ in place of $I_r$.

\subsection{Krawczyk test step}\label{subsec:2}

The main task for this step is proving the existence and uniqueness of the solution path, i.e.~$(x(t), t) \in I\times [t_0,t_1]$ for all $t\in [t_0, t_1],$ with the interval box $I$ and $t_0, t_1$ obtained from the previous step. We establish a parametric version of the Krawczyk method to certify all points in a certain path defined on some closed time interval. 

Let us consider a homotopy $H(x,t):\mathbb{C}^n\times [0,1]\rightarrow \mathbb{C}^n$ with a parameter $t\in [0,1]$ and a solution path $x(t)$. For an $n$-dimensional interval box $I$ in $\mathbb{C}^n$ and an interval $T \subset [0,1]$, the parametric Krawczyk method applies the Krawczyk method on $I$ to the interval extension $\square H(x,T)$, which is obtained by evaluating $H(x,t)$ on $T$ only for $t$ variable. 
The results of the Krawczyk method with parameters are summarized in the theorem below:

\begin{thm}\label{thm:parametric-krawczyk}
    Let $H(x,t):\mathbb{C}^n\times [0,1]\rightarrow \mathbb{C}^n$. Consider intervals $I\subset \mathbb{C}^n$ and $T\subset [0,1]$. For a point $x\in \mathbb{C}^n$ and an $n\times n$-invertible matrix $Y$, define 
    \[K_{x,Y}(I,T):=x-Y\cdot\square H(x,T)+\left(\boldsymbol{1}_n-Y\cdot\square \partial_x H(I,T)\right)\cdot (I-x)\]
    where $\partial_xH$ is the Jacobian of $H$ with respect to $x$ variables.
    Then,
    \begin{enumerate}
        \item if $K_{x,Y}(I,T)\subset I$, then $I$ contains a solution to $H(x,t)$ for each $t\in T$, and 
        \item if additionally $\sqrt{2}\left\|\boldsymbol{1}_n-Y\cdot \square \partial_x H(I,T)\right\|<1$, then $I$ contains a unique solution to $H(x,t)$ for each $t\in T$.
    \end{enumerate}
\end{thm}
\begin{proof}
    For a fixed parameter value $t\in T$, define $F(x):= H(x,t)$. Let $K_{x,Y,F}(I)$ be the Krawczyk operator for $F$ on $I$. In this case, $K_{x,Y,F}(I)\subset K_{x,Y}(I,T)$ and $\boldsymbol{1}_n-Y\cdot \square JF(I)\subset \boldsymbol{1}_n-Y\cdot \square \partial_x H(I,T)$ for any $t\in T$. Applying \Cref{thm:Krawczyk} at each $t\in T$, the result follows.
\end{proof}

To apply the parametric Krawczyk method, we compute the invertible matrix $Y=\partial_x H(x_0,t_0)^{-1}$ and define the time interval $T_{t_0,dt}=[t_0,t_1]$. From the Krawczyk operator $K_{x_0,Y}(I, T_{t_0,dt})$, the existence and uniqueness of $x(t)$ can be verified for all $t\in T_{t_0,dt}$.

\subsection{Successful Krawczyk step}\label{subsec:proceeding}

If the parametric Krawczyk test passes, we proceed to track the solution path $x(t)$ as long as $t_1<1$. To proceed to the next iteration, we set $t_0=t_1$. For a fixed scaling constant $\lambda>1$, we update $dt=\lambda dt$, $r=\lambda r$, and $t_1 = t_0 + dt$. The purpose of scaling is for adaptive choice of both $dt$ and $r$. If the Krawczyk test from the previous step is successful, it may be feasible to proceed with a larger step size $dt$, thereby facilitating rapid path tracking. However, a relatively larger $dt$ compared to $r$ can increase the possibility of failure of the Krawczyk test; hence $r$ should be scaled similarly. The importance of this simultaneous scaling of $dt$ and $r$ is mentioned again in the proof of \Cref{thm:tilted-existence}.

After the scaling of $dt$ and $r$, we apply Newton's method at the midpoint of $I_r$ to update a point $x_0$ approximating the solution $x(t_0)$ to $H(x,t_0)$. After that, repeat the Krawczyk test step.

\subsection{Failed Krawczyk step}\label{subsec:refinement}

There are scenarios in which the Krawczyk test fails. The existence test fails when the solution path deviates from the interval box $I_r$ for some $t\in T_{t_0,dt}$. On the other hand, the uniqueness test might fail if another solution path enters $I$ at some $t\in T_{t_0,dt}$. These scenarios may be resolved by updating $dt=\frac{1}{\lambda}dt$ and $r=\frac{1}{\lambda}r$, and repeating the parametric Krawczyk test.

\subsection{Finalization step}\label{subsec:final}
Assume that the previous Krawczyk test succeeds with an updated value of $t_0\ge 1.$ 
In this case, we refine the midpoint of $I_r$ with the system $H(x,1)=F(x;p_1)$ using Newton's method, and return the refined solution.
The process described in~\Cref{subsec:init,subsec:2,subsec:proceeding,subsec:refinement,subsec:final} is summarized in~\Cref{algo:krawczyk_homotopy_rec}, and illustrated in~\Cref{fig:rect}.

\begin{figure}
\centering
    \begin{tikzpicture}[rotate around x=0,rotate around y=45, rotate around z=0,scale=1]
\draw[line width=0.1mm,color=red] (1,1.5,4.5) -- (2,1.5,4.5);

\draw[line width=0.2mm] (1,.5,5.5) -- (2,.5,5.5);
\draw[line width=0.2mm] (1,.5,5.5) -- (1,2.5,5.5);
\draw[line width=0.2mm] (2,2.5,5.5) -- (1,2.5,5.5);
\draw[line width=0.2mm] (2,2.5,5.5) -- (2,.5,5.5);
\draw (1,1.5,4.5) node {$\bullet$};

\draw[line width=0.2mm] (1,2.5,5.5) -- (1,2.5,3.5);
\draw[line width=0.2mm] (2,2.5,3.5) -- (1,2.5,3.5);
\draw[line width=0.2mm] (2,2.5,3.5) -- (2,2.5,5.5);
\draw[line width=0.2mm] (1,.5,3.5) -- (1,.5,5.5);
\draw[line width=0.1mm,color=red] (2,2,4.5) -- (4,2,4.5);
\draw[<-,line width=0.5mm,dotted,color=red] (2,2,4.5) -- (2,1.5,4.5);
\draw (2,2,4.5) node {$\bullet$};

\draw[line width=0.2mm] (1,.5,3.5) -- (1,2.5,3.5);
\draw[line width=0.2mm] (1,.5,3.5) -- (1,.5,5.5);
\draw[line width=0.3mm,dashed] (2,.5,5.5) -- (2,.5,3.5);
\draw[line width=0.3mm,dashed] (2,2.5,3.5) -- (2,.5,3.5);
\draw[line width=0.3mm,dashed] (2,.5,3.5) -- (1,.5,3.5);

\draw[line width=0.2mm] (2,3.3,3) -- (2,3.3,6);
\draw[line width=0.2mm] (2,.7,6) -- (2,3.3,6);
\draw[line width=0.3mm,dotted] (2,.7,5.5) -- (2,.7,3);
\draw[line width=0.2mm] (2,.7,5.5) -- (2,.7,6);
\draw[line width=0.2mm] (2,3.3,3) -- (2,2.32,3);
\draw[line width=0.3mm,dotted] (2,2.32,3) -- (2,.7,3);

\draw[line width=0.2mm] (2,.7,6) -- (4,0.7,6);
\draw[line width=0.2mm] (2,3.3,6) -- (4,3.3,6);
\draw[line width=0.2mm] (4,0.7,6) -- (4,3.3,6);

\draw[line width=0.2mm] (2,3.3,3) -- (4,3.3,3);
\draw[line width=0.2mm] (4,3.3,6) -- (4,3.3,3);
\draw[line width=0.3mm,dotted] (4,3.3,3) -- (4,.7,3);
\draw[line width=0.3mm,dotted] (4,0.7,6) -- (4,0.7,3);
\draw[line width=0.3mm,dotted] (2,.7,3) -- (4,0.7,3);
\draw[->,line width=0.5mm,dotted,color=red] (4,2,4.5) -- (4,1,4.5);
\draw (4,1,4.5) node {$\bullet$};


\draw[line width=.4mm]
plot[variable=\x,domain=.5:5,samples=73,smooth] 
 (\x,{7/60 *\x^3-23/20 *\x^2+47/15*\x-3/5},4.5);
\draw[line width=.4mm,dotted]
plot[variable=\x,domain=.4:.5,samples=73,smooth] 
 (\x,{7/60 *\x^3-23/20 *\x^2+47/15*\x-3/5},4.5);
\draw[line width=.4mm,dotted]
plot[variable=\x,domain=5:5.5,samples=73,smooth] 
 (\x,{7/60 *\x^3-23/20 *\x^2+47/15*\x-3/5},4.5);

\draw[line width=0.4mm,dotted] (2,.5,5.5) -- (2,0,5.5);
\draw[line width=0.4mm,dotted] (1,.5,5.5) -- (1,0,5.5);
\draw[line width=0.4mm,dotted] (4,.7,6) -- (4,0,6);

\draw[line width=0.3mm,dashed] (1,0,7) -- (1,0,5.5);
\draw[line width=0.3mm,dashed] (2,0,7) -- (2,0,5.5);
\draw[line width=0.3mm,dashed] (4,0,7) -- (4,0,6);

\draw[->,line width=0.1mm] (-1.5,0,7) -- (6.5,0,7);
\draw (6.5,-.1,7.3) node {$t$};

\draw (-1,0,7) node {$\bullet$};
\draw (-1,-.15,7.3) node {$t=0$};

\draw (.9,-.1,7.5) node {$t^{(i-1)}$};
\draw (2.3,-.1,8) node {$t^{(i)}=t^{(i-1)}+dt$};
\draw (4.3,-.1,8) node {$t^{(i+1)}=t^{(i)}+\lambda dt$};
\draw (5.7,0,7) node {$\bullet$};
\draw (5.7,-.15,7.3) node {$t=1$};

\end{tikzpicture}
    \caption{An illustration of \Cref{algo:krawczyk_homotopy_rec}. Solid red lines represent the midpoint of each interval box $I_r$. Dotted red lines show the corrector step producing the next midpoint.}\label{fig:rect}
\end{figure}
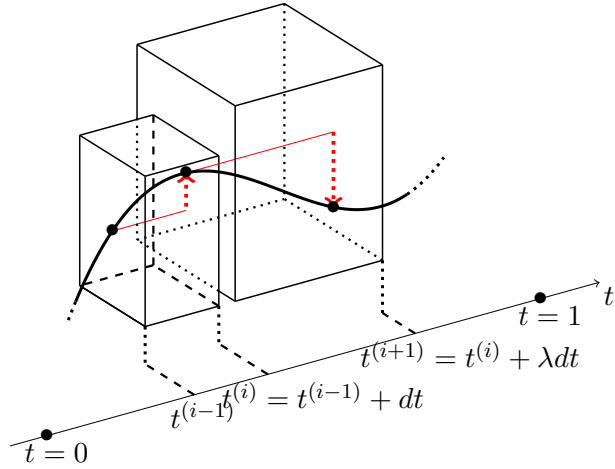

\algrenewcommand\algorithmicrequire{\textbf{Input}:}
\algrenewcommand\algorithmicensure{\textbf{Output}:}
\begin{algorithm}[ht]
	\caption{Krawczyk homotopy}
 \label{algo:krawczyk_homotopy_rec}
\begin{algorithmic}[1]
\Require  
\begin{itemize}
    \item A parameter homotopy $H(x,t)=F(x;p(t)):\mathbb{C}^n\times[0,1]\rightarrow\mathbb{C}^n$ analytic in $x$ and linear in $p$,
    \item a point $x_0$ approximating $x(0)$, for some nonsingular solution path $x(t):[0,1] \to \CC^n$ such that $H(x(t),t)=0$,
    \item a positive number $r>0$ for the initial radius,
    \item a time-step size $dt\in (0,1)$, and
    \item a scaling constant $\lambda>1$.
\end{itemize}
\Ensure A certified approximation of $x(1)$.
\State {Define an interval box $I_r$ centered at $x_0$ with radius $r$.\\
Set $t_0=0,t_1=dt$ and $T_{t_0,dt}=[t_0,t_1]$.\\
Compute $Y:=\partial_x H(x_0,t_0)^{-1}$.}
\While{$t_0<1$}
\State{Run Krawczyk test with $K_{x_0,Y}(I_r,T_{t_0,dt})$.}
\If{Krawczyk test passed}
\State{Set $r=\lambda r$ and $dt=\lambda dt$.}
\State{Set $t_0=t_1$, and $t_1=t_0+dt$.}
\State{Refine the midpoint of $I_r$ with $H(x,t_0)$ to approximate $x(t_0)$ and set it as $x_0$.}
\State{Compute $Y:=\partial_x H(x_0,t_0)^{-1}$.}
\State{Set an interval box $I_r$ centered at $x_0$ with radius $r$, and $T_{t_0,dt}=[t_0,t_1]$.}
\Else
\State{Set $r=\frac{1}{\lambda} r$ and $dt=\frac{1}{\lambda} dt$.}
\State{Set $t_1=t_0+dt$.}
\State{Set an interval box $I_r$ centered at $x_0$ with radius $r$, and $T_{t_0,dt}=[t_0,t_1]$.}
\EndIf
\EndWhile
\State{Refine the midpoint of $I_r$ with $H(x,1)$, and return it.}
 \end{algorithmic}
 \end{algorithm}

\subsection{Preconditioning step (\Cref{algo:preconditioning,algo:krawczyk_homotopy})}

Note that \Cref{algo:krawczyk_homotopy_rec} employs the interval $I\times T$ in a rectangular shape. This approach implicitly assumes that the midpoint of $I_r$ is close enough to the solution path $x(t)$ for all $t\in T_{t_0,dt}$. When the solution path rapidly changes, the algorithm requires frequent reduction of $dt$ and $r$, resulting in slow tracking progress. The preconditioning step discussed in this section adopts more proactive and efficient path prediction for an improved algorithm.   

The preconditioning step from  $t=t_0$ to $t=t_1$ is summarized in~\Cref{algo:preconditioning} below. 
This preconditioning step will be executed before every step involving a Krawczyk test. We assume that an approximation $x_0$ of $x(t_0)$ is known, and $I_r$ is an $n$-dimensional interval box whose midpoint is the origin and radius is $r$. Furthermore, we have $t_1=t_0+dt$ for some $dt$ from the previous step. 

We find a point $x_1$ approximating $x(t_1)$ using the predictor-corrector method. Define the line segment $s(t)$ in $\mathbb{C}^n\times[0,1]$ such that $s(t_0)=x_0$ and $s(t_1)=x_1$. We use this as a prediction of $x(t)$ from $t=t_0$ to $t=t_1$. Compared to \Cref{algo:krawczyk_homotopy_rec}, we define the \emph{tilted interval} to be the Minkowski sum $s(t)+I_r$. 
Just as interval boxes in previous sections used approximate solutions as midpoints, the tilted interval encloses the line segment $s(t).$

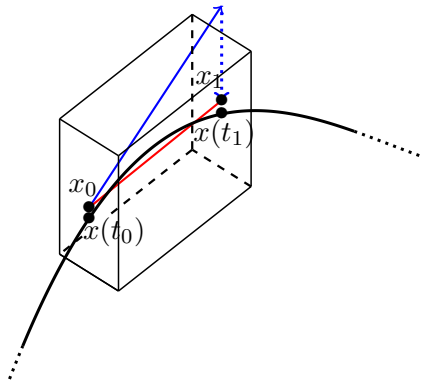
\begin{figure}
    \centering
    \begin{tikzpicture}[rotate around x=0,rotate around y=45, rotate around z=0,scale=1.8]
\draw[->,line width=0.7mm,thick,color=blue] (1,1.58,4.5) -- (2,2.8,4.5);
\draw[<-,line width=0.4mm,dotted,color=blue] (2,2.1,4.5) -- (2,2.8,4.5);
\draw[line width=0.7mm,thick,color=red] (1,1.58,4.5) -- (2,2.1,4.5);

\draw[line width=0.2mm] (1,1.1,5) -- (2,1.6,5);
\draw[line width=0.2mm] (1,1.1,5) -- (1,2.1,5);
\draw[line width=0.2mm] (2,2.6,5) -- (1,2.1,5);
\draw[line width=0.2mm] (2,2.6,5) -- (2,1.6,5);
\draw (1,1.5,4.5) node {$\bullet$};
\draw (1,1.58,4.5) node {$\bullet$};

\draw[line width=0.2mm] (1,2.1,5) -- (1,2.1,4);
\draw[line width=0.2mm] (2,2.6,4) -- (1,2.1,4);
\draw[line width=0.2mm] (2,2.6,4) -- (2,2.6,5);
\draw[line width=0.2mm] (1,1.1,4) -- (1,1.1,5);
\draw (2,2,4.5) node {$\bullet$};
\draw (2,2.1,4.5) node {$\bullet$};

\draw[line width=0.2mm] (1,1.1,4) -- (1,2.1,4);
\draw[line width=0.2mm] (1,1.1,4) -- (1,1.1,5);
\draw[line width=0.3mm,dashed] (2,1.6,5) -- (2,1.6,4);
\draw[line width=0.3mm,dashed] (2,2.6,4) -- (2,1.6,4);
\draw[line width=0.3mm,dashed] (1.17,1.185,4) -- (1,1.1,4);
\draw[line width=0.3mm,dashed] (2,1.6,4) -- (1.36,1.28,4);


\draw[line width=.4mm]
plot[variable=\x,domain=.5:3,samples=73,smooth] 
 (\x,{7/60 *\x^3-23/20 *\x^2+47/15*\x-3/5},4.5);
\draw[line width=.4mm,dotted]
plot[variable=\x,domain=.4:.5,samples=73,smooth] 
 (\x,{7/60 *\x^3-23/20 *\x^2+47/15*\x-3/5},4.5);
\draw[line width=.4mm,dotted]
plot[variable=\x,domain=3:3.5,samples=73,smooth] 
 (\x,{7/60 *\x^3-23/20 *\x^2+47/15*\x-3/5},4.5);

\draw (1, 1.7, 4.4) node {$x_0$};
\draw (1.1, 1.43, 4.7) node {$x(t_0)$};
\draw (2, 2.2, 4.3) node {$x_1$};
\draw (2, 1.87, 4.55) node {$x(t_1)$};

\end{tikzpicture}
    \caption{An illustration of the preconditioning in~\Cref{algo:preconditioning}. The point $x_0$ is an approximation of $x(t_0)$. The blue line represents the predictor step, and the blue dotted line represents the corrector step to get an approximation $x_1$ of $x(t_1)$. The line segment $s(t)$ connecting $x_0$ and $x_1$ is presented by the red line. The tilted interval box is centered at $s(t)$ at each $t\in [t_0,t_1]$ with the same radius.}\label{fig:precon}
\end{figure}

Note that the shape of this tilted interval will be a parallelepiped so that each edge of the interval can be represented by some linear function in $t$. 
Intuitively, tilting seems to offer the advantage of a ``first-order'' approximation of the solution path.
However, applying the Krawczyk method on a tilted interval box might incur significant overestimation due to the conservative nature of interval arithmetic. 
To prevent this issue, we define a new homotopy $\hat{H}(x,t)=H(x+s(t),t)$ which is obtained by change of coordinates via the shearing map $(x, t)\mapsto (x+s(t), t)$. This new homotopy satisfies $\hat{H}(0,t_0)=\hat{H}(0,t_1)=0$. 
In the transformed coordinates, the line segment $s(t)$ is parametrized by $(0,\dots,0,t)$ for $t\in [t_0,t_1]$. We may then apply the parametric Krawczyk method on $I_r$ to $\square\hat{H}(x,[t_0,t_1])$. This step is described in \Cref{algo:preconditioning} and illustrated in \Cref{fig:precon}. After this pre-processing, we conduct the Krawczyk test step.

\begin{algorithm}[ht]
	\caption{Preconditioning}
 \label{algo:preconditioning}
\begin{algorithmic}[1]
\Require  \begin{itemize}
    \item A parameter homotopy $H(x,t)=F(x;p(t)):\mathbb{C}^n\times[0,1]\rightarrow\mathbb{C}^n$ analytic in $x$ and linear in $p$,
    \item a point $x_0$ approximating $x(0)$, for some nonsingular solution path $x(t):[0,1] \to \CC^n$ such that $H(x(t),t)=0$,
    \item a positive number $r>0$ for the radius, and
    \item two positive constants $t_0, t_1\in [0,1]$ with $t_0<t_1$.
\end{itemize}
\Ensure 
\begin{itemize}
    \item A point $x_1$ approximating $x(t_1)$,
    \item a homotopy $\hat{H}(x,t)$, 
    \item an interval box $I_r$, and 
    \item a time interval $T_{t_0,dt}\subset[1,0]$.
\end{itemize}
\State {Find a point $x_1$ approximating $x(t_1)$ using the predictor-corrector method.\\
Compute the line segment $s(t)$ such that $s(t_0)=x_0$ and $s(t_1)=x_1$.\\
Define $\hat{H}(x,t)=H(x+s(t),t)$ so that $\hat{H}(0,t_0)$ and $\hat{H}(0,t_1)$ approximate $0$.\\
Set an interval vector $I_r$ centered at $0$ with the radius $r$, and a time interval $T_{t_0,dt}=[t_0,t_1]$.}
\State {Return $x_1,\hat{H}(x,t), I_r$ and $T_{t_0,dt}$.}
\end{algorithmic}
\end{algorithm}

Compared to the steps discussed in \Cref{subsec:proceeding,subsec:refinement}, there are subtle differences when the preconditioning step is employed. The process of preconditioning involves finding an approximation $x_1$ of $x(t_1)$. Since this process is executed in advance, refining the midpoint of $I_r$ is no longer necessary when proceeding towards larger $t.$
In addition, the preconditioning step must be conducted regardless of the success or failure of the Krawczyk test since $t_1$ must always be updated.
With these caveats, the complete ``tilted'' variant of the Krawczyk homotopy using the preconditioning is described in~\Cref{algo:krawczyk_homotopy}.
See~\Cref{fig:tilt} for an illustration.




\begin{algorithm}[ht]
	\caption{Krawczyk homotopy (tilted)}
 \label{algo:krawczyk_homotopy}
\begin{algorithmic}[1]
\Require 
 \begin{itemize}
    \item A parameter homotopy $H(x,t)=F(x;p(t)):\mathbb{C}^n\times[0,1]\rightarrow\mathbb{C}^n$ analytic in $x$ and linear in $p$,
    \item a point $x_0$ approximating $x(0)$, for some nonsingular solution path $x(t):[0,1] \to \CC^n$ such that $H(x(t),t)=0$,
    \item a positive number $r>0$ for the initial radius,
    \item a time-step size $dt\in (0,1)$, and
    \item a scaling constant $\lambda>1$.
\end{itemize}
\Ensure A certified approximation of $x(1)$.
\State {Set $t_0=0$ and $t_1=dt$.\\
Run \textbf{Preconditioning}$(H(x,t),r,t_0,t_1, x_0)$ to compute $x_1,\hat{H}(x,t),I_r$ and $T_{t_0,dt}$.\\
Compute $Y:=\partial_x \hat{H}(0,t_0)^{-1}$.}
\While{$t_0<1$}
\State{Run Krawczyk test with $K_{0,Y}(I_r,T_{t_0,dt})$.}
\If{Krawczyk test passed}
\State{Set $r=\lambda r$ and $dt=\lambda dt$.}
\State{Refine $x_1$}
\State{Set $x_0=x_1,t_0=t_1$, and $t_1=t_0+dt$.}
\State{Run \textbf{Preconditioning}$(H(x,t),r,t_0,t_1, x_0)$ to compute $x_1,\hat{H}(x,t),I_r$ and $T_{t_0,dt}$.}
\State{Compute $Y:=\partial_x \hat{H}(0,t_0)^{-1}$.
}
\Else
\State{Set $r=\frac{1}{\lambda} r$ and $dt=\frac{1}{\lambda} dt$.}
\State{Run \textbf{Preconditioning}$(H(x,t),r,t_0,t_1, x_0)$ to compute $x_1,\hat{H}(x,t),I_r$ and $T_{t_0,dt}$.}
\EndIf
\EndWhile
\State{Find a point $x_1$ by refining $s(1)$ with the system $H(x,1)$.}
\State{Return $x_1$.}
\end{algorithmic}
\end{algorithm}

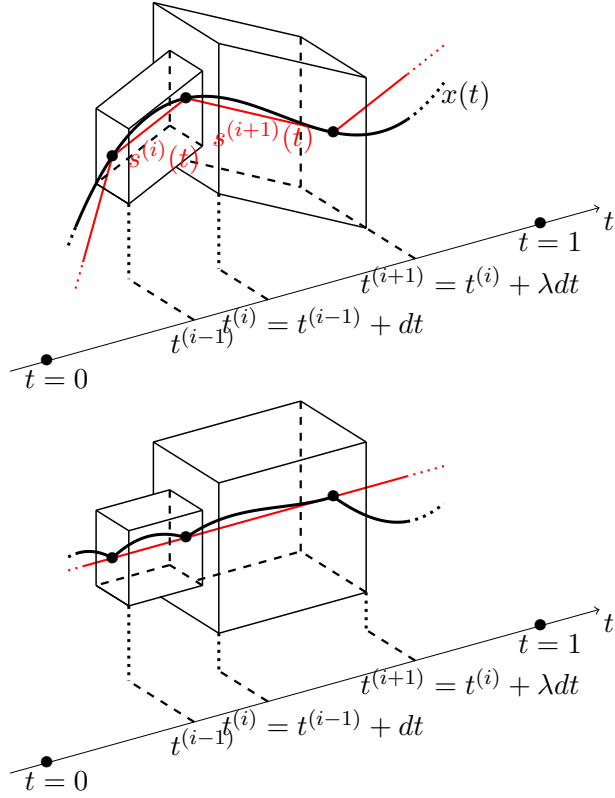
\begin{figure*}
\centering
\begin{tikzpicture}[rotate around x=0,rotate around y=45, rotate around z=0,scale=1]
\draw[line width=0.5mm,thick,color=red] (.6,.2,4.5) -- (1,1.5,4.5);
\draw[line width=0.5mm,thick,color=red,dotted] (.5,-.125,4.5) -- (.6,0.2,4.5);
\draw[line width=0.5mm,thick,color=red] (1,1.5,4.5) -- (2,2,4.5);
\draw[line width=0.5mm,thick,color=red] (4,1,4.5) -- (2,2,4.5);
\draw[line width=0.5mm,thick,color=red] (4,1,4.5) -- (5,1.5,4.5);
\draw[line width=0.5mm,thick,dotted,color=red] (5.5,1.75,4.5) -- (5,1.5,4.5);

\draw[line width=0.2mm] (1,1,5) -- (2,1.5,5);
\draw[line width=0.2mm] (1,1,5) -- (1,2,5);
\draw[line width=0.2mm] (2,2.5,5) -- (1,2,5);
\draw[line width=0.2mm] (2,2.5,5) -- (2,1.5,5);
\draw (1,1.5,4.5) node {$\bullet$};

\draw[line width=0.2mm] (1,2,5) -- (1,2,4);
\draw[line width=0.2mm] (2,2.5,4) -- (1,2,4);
\draw[line width=0.2mm] (2,2.5,4) -- (2,2.5,5);
\draw[line width=0.2mm] (1,1,4) -- (1,1,5);
\draw (2,2,4.5) node {$\bullet$};

\draw[line width=0.2mm] (1,1,4) -- (1,2,4);
\draw[line width=0.2mm] (1,1,4) -- (1,1,5);
\draw[line width=0.3mm,dashed] (2,1.5,5) -- (2,1.5,4);
\draw[line width=0.3mm,dashed] (2,2.5,4) -- (2,1.5,4);
\draw[line width=0.3mm,dashed] (2,1.5,4) -- (1,1,4);

\draw[line width=0.2mm] (2,3,3.5) -- (2,3,5.5);
\draw[line width=0.2mm] (2,1,5.5) -- (2,3,5.5);
\draw[line width=0.2mm] (2,1,4.2) -- (2,1,5.5);
\draw[line width=0.2mm] (2,3,3.5) -- (2,2.2,3.5);

\draw[line width=0.2mm] (2,1,5.5) -- (4,0,5.5);
\draw[line width=0.2mm] (2,3,5.5) -- (4,2,5.5);
\draw[line width=0.2mm] (4,0,5.5) -- (4,2,5.5);

\draw[line width=0.2mm] (2,3,3.5) -- (4,2,3.5);
\draw[line width=0.2mm] (4,2,5.5) -- (4,2,3.5);
\draw[line width=0.3mm,dashed] (4,2,3.5) -- (4,0,3.5);
\draw[line width=0.3mm,dashed] (4,0,5.5) -- (4,0,3.5);
\draw[line width=0.3mm,dashed] (2.5,.75,3.5) -- (4,0,3.5);
\draw (4,1,4.5) node {$\bullet$};


\draw[line width=.4mm]
plot[variable=\x,domain=.5:5,samples=73,smooth] 
 (\x,{7/60 *\x^3-23/20 *\x^2+47/15*\x-3/5},4.5);
\draw[line width=.4mm,dotted]
plot[variable=\x,domain=.4:.5,samples=73,smooth] 
 (\x,{7/60 *\x^3-23/20 *\x^2+47/15*\x-3/5},4.5);
\draw[line width=.4mm,dotted]
plot[variable=\x,domain=5:5.5,samples=73,smooth] 
 (\x,{7/60 *\x^3-23/20 *\x^2+47/15*\x-3/5},4.5);
\draw (5.8,1,4.5) node {$x(t)$};

\draw[line width=0.4mm,dotted] (2,1,5.5) -- (2,0,5.5);
\draw[line width=0.4mm,dotted] (1,1,5) -- (1,0,5);

\draw[line width=0.3mm,dashed] (1,0,7) -- (1,0,5);
\draw[line width=0.3mm,dashed] (2,0,7) -- (2,0,5.5);
\draw[line width=0.3mm,dashed] (4,0,7) -- (4,0,5.5);

\draw[->,line width=0.1mm] (-1.5,0,7) -- (6.5,0,7);
\draw (6.5,-.1,7.3) node {$t$};

\draw (-1,0,7) node {$\bullet$};
\draw (-1,-.15,7.3) node {$t=0$};

\draw (.9,-.1,7.5) node {$t^{(i-1)}$};
\draw (2.3,-.1,8) node {$t^{(i)}=t^{(i-1)}+dt$};
\draw (4.3,-.1,8) node {$t^{(i+1)}=t^{(i)}+\lambda dt$};
\draw (5.7,0,7) node {$\bullet$};
\draw (5.7,-.15,7.3) node {$t=1$};

\draw[color=red] (1.7,1.3,4.5) node {$s^{(i)}(t)$};

\draw[color=red] (3.05,1.2,4.5) node {$s^{(i+1)}(t)$};

\end{tikzpicture}
    \hspace{1pc}
\begin{tikzpicture}[rotate around x=0,rotate around y=45, rotate around z=0,scale=1]
\draw[line width=0.5mm,thick,dotted,color=red] (.4,1.5,4.5) -- (.6,1.5,4.5);
\draw[line width=0.5mm,thick,color=red] (1,1.5,4.5) -- (.6,1.5,4.5);
\draw[line width=0.5mm,thick,color=red] (1,1.5,4.5) -- (2,1.5,4.5);
\draw[line width=0.5mm,thick,color=red] (4,1.5,4.5) -- (2,1.5,4.5);
\draw[line width=0.5mm,thick,color=red] (4,1.5,4.5) -- (5,1.5,4.5);
\draw[line width=0.5mm,thick,dotted,color=red] (5.5,1.5,4.5) -- (5,1.5,4.5);

\draw[line width=0.2mm] (1,1,5) -- (2,1,5);
\draw[line width=0.2mm] (1,1,5) -- (1,2,5);
\draw[line width=0.2mm] (2,2,5) -- (1,2,5);
\draw[line width=0.2mm] (2,2,5) -- (2,1,5);
\draw (1,1.5,4.5) node {$\bullet$};

\draw[line width=0.2mm] (1,2,5) -- (1,2,4);
\draw[line width=0.2mm] (2,2,4) -- (1,2,4);
\draw[line width=0.2mm] (2,2,4) -- (2,2,5);
\draw[line width=0.2mm] (1,1,4) -- (1,1,5);
\draw (2,1.5,4.5) node {$\bullet$};

\draw[line width=0.2mm] (1,1,4) -- (1,2,4);
\draw[line width=0.2mm] (1,1,4) -- (1,1,5);
\draw[line width=0.3mm,dashed] (2,1,5) -- (2,1,4);
\draw[line width=0.3mm,dashed] (2,2,4) -- (2,1,4);
\draw[line width=0.3mm,dashed] (2,1,4) -- (1,1,4);

\draw[line width=0.2mm] (2,.5,5.5) -- (2,2.5,5.5);
\draw[line width=0.2mm] (2,2.5,3.5) -- (2,2.5,5.5);
\draw[line width=0.2mm] (2,.5,3.7) -- (2,.5,5.5);
\draw[line width=0.2mm] (2,2.5,3.5) -- (2,1.8,3.5);

\draw[line width=0.2mm] (2,.5,5.5) -- (4,.5,5.5);
\draw[line width=0.2mm] (2,2.5,5.5) -- (4,2.5,5.5);
\draw[line width=0.2mm] (4,.5,5.5) -- (4,2.5,5.5);

\draw[line width=0.2mm] (2,2.5,3.5) -- (4,2.5,3.5);
\draw[line width=0.2mm] (4,2.5,5.5) -- (4,2.5,3.5);
\draw[line width=0.3mm,dashed] (4,2.5,3.5) -- (4,0.5,3.5);
\draw[line width=0.3mm,dashed] (4,0.5,5.5) -- (4,0.5,3.5);
\draw[line width=0.3mm,dashed] (2.6,.5,3.5) -- (4,0.5,3.5);
\draw (4,1.5,4.5) node {$\bullet$};


\draw[line width=.4mm]
plot[variable=\x,domain=1:2,samples=73,smooth] 
 (\x,{7/60 *(\x)^3-23/20 *(\x)^2+47/15*(\x)-3/5-1/2*\x+.5},4.5);
\draw[line width=.4mm]
plot[variable=\x,domain=2:4,samples=73,smooth] 
 (\x,{7/60 *(\x)^3-23/20 *(\x)^2+47/15*(\x)-3/5+1/2*\x-1.5},4.5);
\draw[line width=.4mm]
plot[variable=\x,domain=4:5,samples=73,smooth] 
 (\x,{7/60 *(\x)^3-23/20 *(\x)^2+47/15*(\x)-3/5-1/2*\x+2.5},4.5);
\draw[line width=.4mm]
plot[variable=\x,domain=.5:1,samples=73,smooth] 
 (\x,{7/60 *\x^3-23/20 *\x^2+47/15*\x-3/5-2*\x+2},4.5);
\draw[line width=.4mm,dotted]
plot[variable=\x,domain=.4:.5,samples=73,smooth] 
 (\x,{7/60 *\x^3-23/20 *\x^2+47/15*\x-3/5-2*\x+2},4.5);
\draw[line width=.4mm,dotted]
plot[variable=\x,domain=5:5.5,samples=73,smooth] 
 (\x,{7/60 *(\x)^3-23/20 *(\x)^2+47/15*(\x)-3/5-1/2*\x+2.5},4.5);

\draw[line width=0.4mm,dotted] (2,.5,5.5) -- (2,0,5.5);
\draw[line width=0.4mm,dotted] (1,1,5) -- (1,0,5);
\draw[line width=0.4mm,dotted] (4,.5,5.5) -- (4,0,5.5);

\draw[line width=0.3mm,dashed] (1,0,7) -- (1,0,5);
\draw[line width=0.3mm,dashed] (2,0,7) -- (2,0,5.5);
\draw[line width=0.3mm,dashed] (4,0,7) -- (4,0,5.5);

\draw[->,line width=0.1mm] (-1.5,0,7) -- (6.5,0,7);
\draw (6.5,-.1,7.3) node {$t$};

\draw (-1,0,7) node {$\bullet$};
\draw (-1,-.15,7.3) node {$t=0$};

\draw (.9,-.1,7.5) node {$t^{(i-1)}$};
\draw (2.3,-.1,8) node {$t^{(i)}=t^{(i-1)}+dt$};
\draw (4.3,-.1,8) node {$t^{(i+1)}=t^{(i)}+\lambda dt$};
\draw (5.7,0,7) node {$\bullet$};
\draw (5.7,-.15,7.3) node {$t=1$};

\end{tikzpicture}

    \caption{A description of \Cref{algo:krawczyk_homotopy}. Thick black curves in each figure represent the solution path of the homotopy. Red line segments in the first figure represent $s^{(i)}(t)$ connecting $x_0$ and $x_1$ for $t\in [t^{(i-1)},t^{(i)}]$. The second figure depicts the situation when the shearing map $x\mapsto x+s^{(i)}(t)$ is applied for $t\in [t^{(i-1)},t^{(i)}]$ at each iteration. The red line in the second figure corresponds to the parametric line segment $(0,\dots, 0,t)$ in $\mathbb{C}^n\times [0,1]$.}\label{fig:tilt}
\end{figure*}

\section{Correctness and termination}\label{sec:proofs}

If the algorithms presented in Section \ref{sec:algorithms} terminate, we obtain both a region and a point within that region such that the point can be refined to an approximation of an exact solution to the system $F(x;p_1)$ to any desired accuracy. 
The correctness of the algorithms is ensured when each interval box $I^{(i)}$ encompasses only one solution path $x(t)$ for all $t\in [t^{(i-1)},t^{(i)}]$. Hence, the proof of \Cref{thm:parametric-krawczyk} also proves the correctness of the algorithms.

To prove termination for a system $F(x;p)$ with parameters $p$, we assume that the parameter homotopy $H(x,t)=F(x;p(t))$ is \emph{affine-linear}; that is, we assume $F$ is affine-linear in the parameters $p$ and $p(t)=(1-t)\cdot p_0+t\cdot p_1$ is a parametric segment. We split the system  $F(x;p)$ into two parts $F(x;p)=F_1(x;p)+F_2(x)$ where $F_1(x;p)$ consists of terms involving parameters while $F_2(x)$ is a collection of terms without parameters (hence, terms only in $x$ variables). For the homotopy $H(x,t)$, we assume a nonsingular solution path $x(t)$ from $t=0$ to $t=1$ and that a starting solution $x(0) \in \CC^n$ is known exactly.
The algorithms are guaranteed to terminate if we can prove that the solution path $x(t)$ from $t=0$ to $t=1$ can be enclosed by a finite collection of interval boxes $I^{(1)}, \dots, I^{(k)}$ constructed by either algorithm.

We first show that the parametric Krawczyk test succeeds in proving the existence and uniqueness of $x(t)$ in an interval $I$ for all values of $t$ in $[t_0,t_0+dt]\subset[0,1]$ when $dt$ and $I$ are small enough. 
Results are presented with the theoretical assumption that the exact solution $x^\star$ is known in advance. However, we also comment on the practical scenario where only an approximation of $x^\star$ is available.

\begin{thm}\label{thm:tilted-existence}
Let $H(x,t):\mathbb{C}^n\times [0,1]\rightarrow \mathbb{C}^n$ be an affine-linear homotopy, analytic with respect to $x$. Assume that we have a point $x^\star\in \mathbb{C}^n$ such that $H(x^\star,t_0)=0$ for some $t_0\in [0,1]$. Consider a fixed positive constant $L>0$ such that $L\geq \|\square \partial_x^2H(I_1,[0,1])\|,$ where $I_1$ is an interval box centered at $x^\star$ of radius $1$. Suppose that the solution path $x(t)$ is nonsingular for all $t\in T_{t_0,dt}:=[t_0,t_0+dt]$ for some $dt>0$ such that $T_{t_0,dt}\subset [0,1]$. 
Then, there exist $0<r<1$ and $0<dt<1$ such that the path $x(t)$ is uniquely contained in the interval box $I_r$ centered at $x^\star$ with the radius $r$ whenever $t\in T_{t_0,dt}$. In other words, if the constant $R=\frac{dt}{r}$ satisfies 
$$\sqrt{n}-\|Y\|\cdot R\cdot \|F_1(x^\star;p_1-p_0)\|>0$$
and 
\begin{equation}\label{eq:krawczyk-bound-3}
    \sqrt{nr^2+dt^2}<\frac{1}{\sqrt{2}\cdot\|Y\|\cdot L},
\end{equation} then $K_{x^\star,Y}(I_r,T_{t_0,dt})\subset I_r$ and
$$\|\boldsymbol{1}_n-Y\cdot \square \partial_xH(I_r,T_{t_0,dt})\|\leq \frac{1}{\sqrt{2}}.$$
\end{thm}
\begin{proof}
Take any $y\in K_{x^\star,Y}(I_r,T_{t_0,dt})$. Then, from the definition of the Krawczyk operator $K_{x,Y}$, we have
\begin{equation}\label{eq:krawczyk-goal}
y-x^\star\in-Y\cdot\square H(x^\star,T_{t_0,dt})+(\boldsymbol{1}_n-Y\cdot \square \partial_xH(I_r,T_{t_0,dt}))\cdot[-r,r]^n.
\end{equation}
Since $x^\star$ is the midpoint of $I_r$, our goal is to show $\|y-x^\star\|\leq r$.
Note that we choose $Y=\partial_xH(x^\star,t_0)^{-1}$. Therefore, 
\begin{align}
\|\boldsymbol{1}_n-Y\cdot \square \partial_xH(I_r,T_{t_0,dt})\| & = \left\|Y\cdot \Big(\partial_xH(x^\star,t_0)-\square \partial_xH(I_r,T_{t_0,dt})\Big)\right\| \nonumber \\
&\leq \|Y\|\cdot L\cdot \sqrt{nr^2+dt^2} \label{eq:krawczyk-bound-1}
\end{align}
if $r<1$ and $0<dt<1$.
The last inequality follows from the differentiability of $H$ and the Lipschitz continuity of $\partial_xH$ \cite[Section 1.5, Theorem 1.3]{1130000794961932160}. 

Because $H(x^\star,t_0)=0$, we have for any $\delta \in [0, dt]$ that
\begin{align*}
    H(x^\star,t_0+\delta)&=F(x^\star;p(t_0+\delta))\\
    &=F(x^\star;(1-t_0-\delta)\cdot p_0+(t_0+\delta)\cdot p_1)\\
    &=F(x^\star; (1-t_0)\cdot p_0 + t_0\cdot p_1)
    +
    F(x^\star ; \delta \cdot(p_1 - p_0))\\
    &= H(x^\star , t_0) + F_1\left(x^\star;\delta\cdot(p_1-p_0)\right)\\
    &=F_1\left(x^\star;\delta\cdot(p_1-p_0)\right).
\end{align*}

Therefore, we know that 
\begin{equation}\label{eq:krawczyk-bound-2}
    \|\square H(x^\star,T_{t_0,dt})\|\leq dt\cdot\|F_1(x^\star;p_1-p_0)\|.
\end{equation}
Using equations~\eqref{eq:krawczyk-bound-1},~\eqref{eq:krawczyk-bound-2})
to bound points in the interval of~\eqref{eq:krawczyk-goal}, we deduce that $K_{x^\star,Y}(I_r,T_{t_0,dt})\subset I_r$ if there are positive $r$ and $dt$ satisfying 
\[
\|Y\|\cdot dt\cdot \|F_1(x^\star;p_1-p_0)\|+\|Y\|\cdot L\cdot\sqrt{nr^2+dt^2}\cdot 2r\sqrt{n}\leq r.\]

Setting $R=\frac{dt}{r}$, we rewrite the inequality above as
\ifthenelse{\mainfilecheck{1} > 0}
{\begin{equation}\label{eq:Y-bound}
\left(\|Y\|\cdot R\cdot \|F_1(x^\star;p_1-p_0)\|-1\right)r+\\\left(\|Y\|\cdot L\cdot \sqrt{n+R^2}\cdot 2 \sqrt{n} \right)r^2\leq 0.
\end{equation}}
{\begin{multline}\label{eq:Y-bound}
\left(\|Y\|\cdot R\cdot \|F_1(x^\star;p_1-p_0)\|-1\right)r+\\\left(\|Y\|\cdot L\cdot \sqrt{n+R^2}\cdot 2 \sqrt{n} \right)r^2\leq 0.
\end{multline}}
This inequality is satisfied for a positive value of $r$ provided that $R$ is sufficiently small.
More precisely, choosing $R$ small enough that
\begin{equation}\label{eq:R-bound}
 1-\|Y\|\cdot R\cdot \|F_1(x^\star;p_1-p_0)\|>0,   
\end{equation} some positive $r$ satisfying the inequality (\ref{eq:Y-bound}) exists. 
This concludes the existence statement that $x(t)\in I_r$ for any $t\in T_{t_0, dt}.$ 

Lastly, using~\eqref{eq:krawczyk-bound-1}, we have $\sqrt{nr^2+dt^2}<\frac{1}{\sqrt{2}\cdot\|Y\|\cdot L}$ for suitably small $r$ and $dt$. This proves uniqueness of the solution path in $I_r.$
\end{proof}

Note that the theorem and its proof applies to both \Cref{algo:krawczyk_homotopy_rec,algo:krawczyk_homotopy}. When considering the case of \Cref{algo:krawczyk_homotopy}, the statement is relevant to the homotopy $\hat{H}(x,t)$ rather than $H(x,t)$. 

\begin{remark}\label{rem:inexact}
    Recall that~\Cref{thm:parametric-krawczyk} is stated for the exact solution $x^\star$. In practice, having $x^\star$ is not feasible, but we will have an approximation $x_0$. Let $I_r$ be the interval centered at $x_0$ with the radius $r$, and assume that $\|H(x_0,t_0)\|\leq \epsilon$. In this case, the inequality (\ref{eq:krawczyk-bound-1}) still holds. On the other hand, we have
    \begin{align*}
    H(x_0,t_0+\delta)&=F(x_0;p(t_0+\delta))\\
    &=F(x_0;(1-t_0-\delta)\cdot p_0+(t_0+\delta)\cdot p_1)\\
    &=F(x_0; (1-t_0) \cdot p_0 +t_0\cdot p_1)
    +
    F(x_0 ; \delta\cdot (p_1 - p_0))\\
    &= H(x_0, t_0) + F_1\left(x_0;\delta\cdot(p_1-p_0)\right).
\end{align*}
Hence, the inequality (\ref{eq:krawczyk-bound-2}) becomes
\[\|\square H(x_0,T_{t_0,dt})\|\leq \epsilon + dt\cdot \|F_1(x_0;p_1-p_0)\|\]
Introducing $R=\frac{dt}{r}$, the inequality (\ref{eq:Y-bound}) turns into
\ifthenelse{\mainfilecheck{1} > 0}
{\begin{multline*}
\|Y\|\cdot\epsilon+\left(\|Y\|\cdot R\cdot \|F_1(x_0;p_1-p_0)\|-1\right)r\\+\left(\|Y\|\cdot L\cdot \sqrt{n+R^2}\cdot 2 \sqrt{n} \right)r^2\leq 0.
\end{multline*}}
{\begin{multline*}
\|Y\|\cdot\epsilon+\left(\|Y\|\cdot R\cdot \|F_1(x_0;p_1-p_0)\|-1\right)r\\+\left(\|Y\|\cdot L\cdot \sqrt{n+R^2}\cdot 2 \sqrt{n} \right)r^2\leq 0.
\end{multline*}}
For small enough $\epsilon ,$ note that if inequality (\ref{eq:R-bound}) is satisfied, the positive values of $r$ satisfying this inequality exist and are bounded below. Hence, when replacing the interval box $I_r$, it is crucial to select one with a radius that is not excessively small. A careful discussion of refining an interval box to pass the Krawczyk test appears in~\cite{guillemot2024validated}.
\end{remark}

We provide a corollary proving the termination of the algorithms. The goal is finding a uniform lower bound for $dt$ and $r$ so that the algorithm terminates in finitely many iterations.

\begin{corollary}\label{cor:termination}
With the same hypotheses on $H(x,t)$ as in~\Cref{thm:tilted-existence},   \Cref{algo:krawczyk_homotopy} terminates in finitely many steps. 
\end{corollary}
\begin{proof}
Our smoothness assumption implies the real-valued function $\rho (t):[0,1] \to \mathbb{R}$ defined by $\rho (t) = \| \partial_x H (x(t), t))^{-1} \|$ is uniformly bounded by some constant $M_1$.
During any particular iteration at time $t_0,$ we have $\| Y \| = \| \partial_x\hat{H} (0, t_0)^{-1} \| = \rho (t_0) \le M_1.$
Furthermore, considering a sufficiently large compact region that contains $x^\star$, we know that $\|F_1(x^\star,p_1-p_0)\|$ can be bounded uniformly by some constant, $\|F_1(x^\star,p_1-p_0)\|\le M_2$. 

Finally, defining $I_r$ to be the interval box in $\mathbb{C}^n$ centered at the origin with the radius $r$, we claim that for some fixed $\hat{r},\delta>0$, there is a constant $M_3>0$ such that $M_3\geq \|\square \partial_x^2\hat{H}(I_{\hat{r}},[t_0,t_0+\delta])\|$ at any iteration at $t_0\in [0,1]$.
By considering a sufficiently large compact region containing $x(t)$, we know that there is $L>0$ such that $L\geq \|\partial_x^2H(x(t),t)\|$ for any $t\in [0,1]$. Therefore, there is $\hat{r}>0$ such that $2L\geq\|\partial_x^2 H(x,t)\|$ for any $x\in x(t)+I_{2\hat{r}}$ and $t\in[0,1]$. Here, $x(t)+I_{2\hat{r}}$ is the Minkowski sum of $x(t)$ and $I_{2\hat{r}}$. 
Then, there is $\delta>0$ such that for any $t_0\in [0,1]$, the line segment $s(t)$ connecting $x(t_0)$ and $x(t_0+dt)$ is contained in $x(t)+I_{\hat{r}}$ whenever $dt\leq\delta$. Note that this is possible since we may assume that $\|x'(t)\|<\infty$ for all $t\in[0,1]$. Setting $\hat{H}(x,t)=H(x+s(t),t)$, we have $\|\partial_x^2\hat{H}(x,t)\|\leq 2L$ if $x\in I_{\hat{r}}$ and $t\in [t_0,t_0+\delta]$ at any time $t_0\in [0,1]$. This shows that $M_3=2L$ is a uniform upper bound on $\|\square \partial_x^2 \hat{H}(I_{\hat{r}},[t_0,t_0+\delta])\|$ at any $t_0\in[0,1]$.





We now apply~\Cref{thm:tilted-existence} to the homotopy $\hat{H}.$
Referring to~\eqref{eq:R-bound}, the existence test succeeds provided that
\[
R < (M_1 \cdot M_2)^{-1} 
\Rightarrow 
dt < r \cdot (M_1\cdot M_2)^{-1} < (M_1 \cdot M_2)^{-1}. 
\]
Thus, there is a uniform lower bound on the value of $dt $ at any point in the algorithm.
Moreover, by the bound \eqref{eq:krawczyk-bound-3}, choosing $r$ and $dt$ satisfying that
\[\sqrt{nr^2+dt^2}<\frac{1}{\sqrt{2}\cdot M_1\cdot M_3},\]
uniform lower bounds for $r$ and $dt$ are obtained.
Therefore, the uniqueness of the solution path when tracking from an exact solution at time $t=0$ is guaranteed.
When tracking from an approximate solution at subsequent times $t=t^{(1)}, \ldots ,$ we proceed by first refining the solution so that $r$ can be chosen as indicated in~\Cref{rem:inexact}, giving us the needed analogue of inequality~\eqref{eq:Y-bound}.
\end{proof}

For the base-line method~\Cref{algo:krawczyk_homotopy_rec}, the proof is similar, but simpler, since~\Cref{thm:parametric-krawczyk} applies directly to the homotopy $H.$

\begin{remark}\label{remark:termination}
The theoretical analysis of~\Cref{algo:krawczyk_homotopy,algo:krawczyk_homotopy_rec} depends on carefully-chosen constants.
We point out the following regarding these constants:
    \begin{enumerate}
\item    From the inequality~\eqref{eq:R-bound}, a sufficiently small value of $R$ is required to guarantee termination. As~\Cref{algo:krawczyk_homotopy_rec,algo:krawczyk_homotopy} do not change the value of $R$, the initial values of $dt$ and $r$ should be chosen carefully to ensure $R$ is sufficiently small. The resulting theoretical bounds on $R$ are likely pessimistic, which would lead to more iterations per path. Thus, in practice, it is recommended to choose a ``reasonable'' initial value for $R$, at the possible expense of termination.
    Our experiments use a range $R \in [1/5, 2].$
\item Additional analysis is needed in order to make all constants appearing in~\Cref{cor:termination} effective. For instance, a uniform bound $|\rho (t)|\le M_1$ for all $t\in [0,1]$ can be obtained under the slightly stronger hypothesis that $x(t)$ is analytic on an open interval containing $[0,1].$
Indeed, since the value of the solution curve $x(0)$ is known exactly, and the values of derivatives of $x(t)$ can be computed using the Davidenko differential equation $x'(t) = (\partial_x H)^{-1}\cdot \partial_t H,$ this hypothesis would enable a uniform Taylor approximation of $\rho (t)$ to any desired accuracy on $[0,1].$
\end{enumerate}
\end{remark}

\section{Experiments}\label{sec:experiments}

In this section, we present experiments conducted with our preliminary implementation of \Cref{algo:krawczyk_homotopy} in Macaulay2~\cite{M2}.
Throughout this section, we use hyper-parameter setting $\lambda = 3$ for the step increase/decrease factor. The values for step-size $dt$ and radius $r$ depend on the experiment.
Real interval arithmetic computations are performed by the library MPFI~\cite{mpfi}.
A current limitation of our implementation is that complex interval computations are performed at the top-level, and thus tracking complex homotopies is slower than real homotopies.

For parametric systems $F(x; p)$ defined over the real numbers, we are typically interested in real-valued solutions $x\in \RR^n$ for real-valued parameters $p\in \RR^m$.
However, a real parameter path $p:[0,1] \to \RR^m$ typically leads to singularities: that is, $JF (x(t); p(t))$ will be singular for some $t\in (0,1).$
Thus, it is typical to instead use a complex-valued path $p: [0,1] \to \CC^m$, whose target parameters $p_1=p(1)$, and possibly also start parameters $p_0 = p(0),$ are real-valued.
Moreover, $p$ is constructed in a randomized fashion; depending on the application, $p_1\in \CC^m$ may be chosen randomly, or for $p_1\in \RR^m$ a suitably random complex path may be constructed using the $\gamma$-trick~\cite[Chapter 8]{SommeseWampler:2005} or some variant thereof.

Motivated by the preceding discussion, we evaluate our method for complex homotopies in~\Cref{subsec:complex-comparison} by comparing the number of predictor steps used by our method to those reported in previous works on examples ranging with $1\le n \le 6$ variables.
Complementary to these results, we present timings for tracking a special class of real homotopies in~\Cref{subsec:real-comparison} in up to $n=20$ variables.
Taken together, these results show that the Krawczyk homotopy is competitive with the previous state-of-the-art in certified path tracking, and that the number of variables is not an inherent limitation.

All experiments were conducted with a Macbook M2 pro 3.5 GHz, 16 GB RAM. The code is available at 
\begin{center}
\url{https://github.com/klee669/krawczykHomotopy}
\end{center}

\subsection{Benchmark examples}\label{subsec:complex-comparison}

We begin with the univariate example presented in \cite[Section 7.1]{hauenstein2014posteriori}. Considering $F(x)=x^2-1-m$ and $v=m$ for $m>-1$, we define a homotopy $H(x,t)=F(x)+v t=x^2-1-m+mt$. For the initial choice of $dt=.02$ and $r=.1$, we measure the number of iterations by varying the value of $m$. The result is summarized in \Cref{tab:comparison-newtonhomotopy}. 
For given initial values of $dt=.02$ and $r=.1$, \Cref{algo:krawczyk_homotopy} requires fewer iterations than that of \cite{hauenstein2014posteriori} except for $m=30000$. We remark that, depending on initial values of $dt$ and $r$, the number of iterations of \Cref{algo:krawczyk_homotopy} may vary.

We also consider benchmark examples of~\cite{beltran2012certified} using our implementation with initial values $dt=.1$ and $r=.1$. For each example, we measure the maximum, minimum, and average number of iterations and compare the results with those reported in \cite{beltran2012certified} (See \Cref{table:comparisonResults}).

Lastly, to examine the impact of values of the hyper-parameters $dt$ and $r$ on the performance of the Krawczyk homotopy, we address two benchmark problems using different values of $dt$ and $r$ while varying their ratios (See \Cref{table:ratioResults}). The table shows that the results have similar average numbers of iterations for a given ratio $R=\frac{dt}{r}$ of $dt$ and $r$. The results suggest that the performance of the Krawczyk homotopy is more significantly influenced by the ratio of $dt$ to $r$, rather than their individual values.

\subsection{A real homotopy}\label{subsec:real-comparison}

In addition to the results obtained for complex homotopies, we provide timings for tracking a special class of real homotopies where bifurcations are naturally avoided.
Our setup is the classical problem of \emph{low-rank matrix approximation}, following the geometric formulation in~\cite{draisma2016euclidean}.
Let $\mathcal{V}_1 \subset \RR^{n\times n}$ denote variety of rank $\le 1$ matrices, and consider the incidence correspondence
\[
\mathcal{E}_n = \left\{ (A, xy^T) \in \RR^{n \times n} \times \mathcal{V}_1 \mid (A - x y^T) \in \left(T_{xy^T} \mathcal{V}_1 \right)^\perp \right\} 
\]
(here $T_\bullet$ denotes the tangent space).
The map $\pi : \mathcal{E}_n \to \RR^{n\times n}$ onto the first factor is a generically $n$-to-$1$ map. More precisely, for generic $A\in \RR^{n \times n}$ we have from the singular value decomposition, $$A = \left( u_1 \, | \cdots | \, u_n\right) \operatorname{diag} (\sigma_1, \ldots , \sigma_n) 
\left( v_1 \, | \cdots | \, v_n\right)^T, \quad \text{that}
$$
\[\pi^{-1} (A) = \left\{ (\sqrt{\sigma_1} u_1) (\sqrt{\sigma_1} v_1)^T, \,\ldots , \,  (\sqrt{\sigma_n} u_n) (\sqrt{\sigma_n} v_n)^T \right\}.\] 
Moreover, it is known that the branch locus of $\pi $---defined here to be the set of points $A \in \RR^{n\times n}$ such that $| \pi^{-1} (A)| \ne n$---has codimension greater than one~\cite{ilyushkin}.
Because of this, we may expect that a suitably parameter path $p(t): [0,1] \to \RR^{n\times n}$, with $p(0)= A_0,$ $p(1) = A_1$, avoids the branch locus with probability-one, and use such paths to construct homotopies connecting known points of one fiber $\pi^{-1} (A_0)$ to another $\pi^{-1} (A_1)$.

In our experiments, we consider the straight-line segment $p$ that connects the identiy matrix $A_0 = I$ to the \emph{Hilbert matrix},
\[
A_1 = \begin{pmatrix}
1 & \frac{1}{2} & \cdots & \frac{1}{n-1}\\
\frac{1}{2} & \frac{1}{3} & \cdots & \frac{1}{n} \\
\vdots & \vdots & \ddots & \vdots \\
\frac{1}{n} & \frac{1}{n+1} & \cdots & \frac{1}{2n-1}
\end{pmatrix}.
\]
This is a notoriously ill-conditioned test matrix used in numerical linear algebra.
Our modest goal is to certify the paths connecting the best rank-one approximations, given by $(\sqrt{\sigma_1} u_1) (\sqrt{\sigma_1} v_1)^T,$ which are the easiest solution paths to track in this example.

To carry this out, we use a suitable system of $2n$ parametric equations in variables $x,y\in \RR^n$ which vanish on $\mathcal{E}_1$ and whose solution paths are regular throughout the homotopy.
From the objective function $\ell (x,y; A)= \| A - x y^T\|_2^2,$ we use $2n-1$ critical point equations $\partial_{x_2} g, \ldots , \partial_{x_n} g, \partial_{y_1} g, \ldots , \partial_{y_n} g,$ and impose the equation of a generic affine chart $b^T x - c = 0.$

The results of our experiments with this real homotopy are shown in~\Cref{table:matrix-euc}.
The table illustrates that the number of steps used by our method grows moderately with respect to the number of variables.
The measured timings grow at a comparable rate.
The number of iterations per second for this example is seen to steadily decrease with the number of variables.

The changes in the step-size $dt$ as tracking progresses are visualized in~\Cref{fig:dtgraph}. For this problem,~\Cref{algo:krawczyk_homotopy} requires smaller step-size both as $t$ tends towards $1$ and as $n$ increases. 

We point out that there will be some variance in such experiments due e.g.\ to the randomly chosen chart, as witnessed by the progression between cases $n=5,6,7$.
Still, even for the cases considered with tens of variables, the step-size consistently stays above unit roundoff and all paths are successfully tracked within an hour.

\section*{Acknowledgements}
We thank Michael Burr for several useful discussions.
We also thank the ISSAC 2024 referees for many helpful comments that led to improvements in the manuscript.
Timothy Duff acknowledges support from an NSF Mathematical Sciences Postdoctoral Research Fellowship (DMS-2103310).

\section*{Additional figures and tables}
\begin{table}[ht]
    \centering
$dt=.02, r =.1\qquad\qquad\qquad\qquad\quad$\\
\begin{tabular}{c|c||c}
    $m$ value & $\# $ iters & HHL \cite{hauenstein2014posteriori} $\#$ iters\\
    \hhline{=|=||=}
    $10$ & $31$ & $51$ \\
    $40$ & $20$ & $82$ \\
    $100$ & $14$ & $105$ \\
    $2000$ & $23$ & $180$ \\
    $5000$ & $57$ & $204$ \\
    $10000$ & $108$ & $220$ \\
    $30000$ & $327$ & $250$ \\
    \end{tabular}    \caption{Comparing the number of iterations between \Cref{algo:krawczyk_homotopy} and the certified tracking algorithm in \cite{hauenstein2014posteriori}.}
    \label{tab:comparison-newtonhomotopy}
\end{table}

\begin{table}[ht]
$dt=.1, r =.1\qquad\qquad\qquad\qquad\qquad\qquad\qquad\qquad\qquad\qquad\quad$\\
\centering
\scalebox{.988}{
\begin{tabular}{c|c|r|c|c||r}
    System & $\#$ roots & avg.~$\#$ iters & max &min & BL \cite{beltran2012certified} avg. \\
    \hhline{=|=|=|=|=||=}
    Random$_{(2^3)}$ & $8$ & $317.75$ & $509$ & $203$ & $198.5$\\
    Random$_{(2^4)}$ & $16$ & $563.25$ & $1257$ & $211$ & $813.81$\\
    Random$_{(2^5)}$ & $32$ & $675.31$ & $5119$ & $209$ & $1542.5$\\
    Random$_{(2^6)}$ & $64$ & $1166.42$ & $5415$ & $267$ & $2211.58$\\
    Katsura3 & $4$ & $264.25$ & $289$ & $239$ & $569.5$\\
    Katsura4 & $8$ & $331.75$ & $451$ & $233$ & $1149.88$\\
    Katsura5 & $16$ & $444.75$ & $731$ & $311$ & $1498.38$\\
    Katsura6 & $32$ & $721.47$ & $1391$ & $497$ & $2361.81$
\end{tabular}}
\caption{\Cref{algo:krawczyk_homotopy} versus certified tracking in~\cite{beltran2012certified}.}
\label{table:comparisonResults}
\end{table}

\begin{table}[ht]
Random$_{(2^3)}\qquad\qquad\qquad\qquad\qquad\qquad\qquad\qquad\qquad\quad$\\
{\centering
\begin{tabular}{c||c|r||c|r}
    $R=\frac{dt}{r}$ & $(dt,r)$ & avg. $\#$ iters &     $(dt,r)$ & avg. $\#$ iters  \\[1.2pt]
    \hhline{=||=|=||=|=}
    $.5$ & $(.2,.4)$ & $134.75$ &      $(.02,.04)$ & $135.63$\\
     $1$ &   $(.4,.4)$ & $96.5$ &     $(.04,.04)$ & $99.5$  \\    $1.5$ & $(.6,.4)$ & $98.63$ & $(.06,.04)$ & $98.88$ \\
    $2$ & $(.8,.4)$ & $96.75$  &    $(.08,.04)$ & $99.5$ 
    \end{tabular}}\\
    \vspace{5pt}
Katsura4$\qquad\qquad\qquad\qquad\qquad\qquad\qquad\qquad\qquad\qquad$\\
{\centering
\begin{tabular}{c||c|r||c|r}
    $R=\frac{dt}{r}$ & $(dt,r)$ & avg. $\#$ iters &    $(dt,r)$ & avg. $\#$ iters  \\[1.2pt]
    \hhline{=||=|=||=|=}
$.5$ & $(.2,.4)$ & $287.75$ &  $(.02,.04)$ & $298.88$ \\
$1$ & $(.4,.4)$ & $351.88$ & $(.04,.04)$ & $326.25$ \\
$1.5$ & $(.6,.4)$ & $444.38$ &        $(.06,.04)$ & $460.13$  \\
 $2$ & $(.8,.4)$ & $594.5$  & $(.08,.04)$ & $552.63$ \\
\end{tabular}}\\
\caption{Average number of iterations for different values of the ratio $\frac{dt}{r}.$}
\label{table:ratioResults}
\end{table}

\begin{table}[ht]
\centering
$dt=.2, r =.1\qquad\qquad\qquad\qquad\qquad\quad$\\
\begin{tabular}{c|c|c|c}
    $n$ & $\#$ vars &  $\#$ iters & elapsed time (s.)\\
    \hhline{=|=|=|=}
    2 & 4& 37 & 1.91\\
    3 & 6&  147 & 10.77\\
    4& 8& 291 & 32.55\\
    5& 10 &421 & 70.07\\
    6& 12 &885 & 215.72\\
    7& 14 & 1067 & 364.56\\
    8& 16& 1808 & 820.29\\
    9& 18&2119 & 1294.18\\
    10& 20& 3611 & 2768.96

\end{tabular}
\caption{Timing data for low-rank matrix approximation.}
\label{table:matrix-euc}
\end{table}
 \begin{figure}[ht]
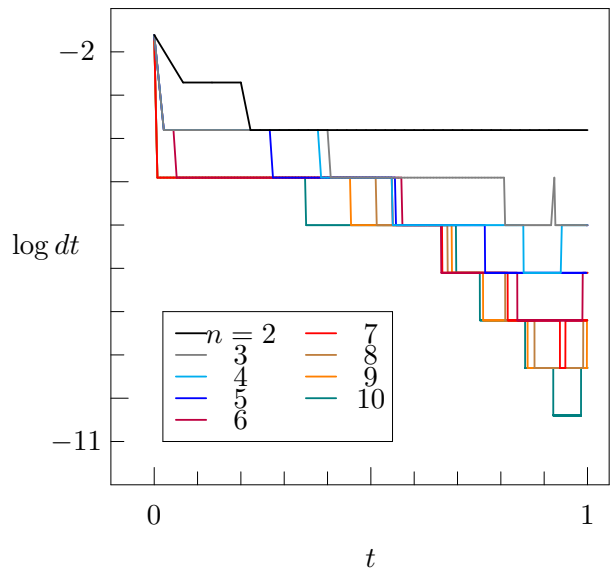

     \centering
     \include{dtgraph} 
     \caption{Step-sizes (log scale) for low-rank approximation.}
     \label{fig:dtgraph}
 \end{figure}

\bibliography{ref.bib}
\bibliographystyle{abbrv}

\end{document}